\documentclass[runningheads,envcountsame,numbook]{svjour3}%
\usepackage{amsmath,amssymb}
\usepackage{cite}
\usepackage[left=4.5cm, right=4.5cm, top=4cm, bottom=4cm]{geometry}
\usepackage{color}
\usepackage{hyperref}
\usepackage[colorinlistoftodos]{todonotes}
\usepackage[weather,alpine,misc,geometry]{ifsym}
\usepackage{amsmath}
\usepackage{amsfonts}
\usepackage{amssymb}
\usepackage{graphicx}%
\providecommand{\U}[1]{\protect\rule{.1in}{.1in}}
\hypersetup{
    colorlinks=true,
    linkcolor=blue,     citecolor=red,     urlcolor=blue  }
\smartqed

\newcommand {\R} {\mathbb R}
\newcommand {\N} {\mathbb N}
\def\es{\varnothing}
\newcommand{\al}{\alpha}
\newcommand {\dom} {{\rm dom}\,}
\newcommand {\epi} {{\rm epi}\,}
\newcommand {\gph} {{\rm gph}\,}
\newcommand{\AK}[1]{\todo[inline]{AK {#1}}}
\newcommand{\iv}{^{-1} }

\newcommand{\la}{\lambda}

\newcommand{\bx}{\bar x}
\newcommand{\by}{\bar y}
\newcommand{\ang}[1]{\left\langle #1 \right\rangle}
\newcommand{\qdtx}[1]{\quad\mbox{#1}\quad}
\def\lsc{lower semicontinuous}
\newcommand {\sd} {\partial}
\def\RHS{right-hand side}
\def\LHS{left-hand side}
\newcounter{mycount}
\def\cnta{\setcounter{mycount}{\value{enumi}}}
\def\cntb{\setcounter{enumi}{\value{mycount}}}

\newcommand {\Erq} {{\rm er_q}\,}
\newcommand {\clmq} {{\rm clm_q}\,}
\newcommand {\srgq} {{\rm srg_q}\,}
\newcommand {\srpq} {{\rm shrp_q}\,}
\begin{document}

\title{Isolated calmness and sharp minima via H\"older graphical derivatives
\thanks{The research of the first and second authors was supported by the
Australian Research Council, project DP160100854. The first author benefited
from the support of the European Union's Horizon 2020 research and innovation
programme under the Marie Sk{\l }odowska--Curie Grant Agreement No. 823731
CONMECH. The research of the second author was also supported by MINECO of
Spain and FEDER of EU, grant MTM2014-59179-C2-1-P. The research of the third
author was supported by the Research Grants Council of Hong Kong (RGC Ref No.
15216518). The research of the fourth author was supported by the National
Natural Science Foundation of China, grants 11801497 and 11771384. } }
\author{Alexander Y. Kruger
\and Marco A. L\'{o}pez
\and Xiaoqi Yang
\and Jiangxing Zhu }

\institute{
Alexander Y. Kruger \at
Centre for Informatics and Applied Optimization, Federation University Australia, Ballarat, Australia\\
\email{a.kruger@federation.edu.au}
\and
Marco A. L\'{o}pez (\Letter\,) \at
Department of Mathematics, University of Alicante, Alicante, Spain, and Centre for Informatics and Applied Optimization, Federation University Australia, Ballarat, Australia\\
\email{marco.antonio@ua.es}
\and
Xiaoqi Yang \at
Department of Applied Mathematics, The Hong Kong Polytechnic University, Hong Kong\\
\email{xiao.qi.yang@polyu.edu.hk}
\and
Jiangxing Zhu \at
Department of Mathematics, Yunnan University, Kunming 650091, China\\
\email{jiangxingzhu@yahoo.com}
} \dedication{Dedicated to the 70th birthday of our colleague and friend Miguel Goberna}
\maketitle

\begin{abstract}
The paper utilizes H\"{o}lder graphical derivatives for characterizing
H\"older strong subregularity, isolated calmness and sharp minimum. As
applications, we characterize H\"older isolated calmness in linear
semi-infinite optimization and H\"older sharp minimizers of some penalty
functions for constrained optimization.
\end{abstract}




\section{Introduction}

This paper continues our previous work \cite{KruLopYanZhu19} and utilizes
H\"{o}lder graphical derivatives (sometimes referred to as Studniarski
derivatives) for characterizing certain regularity properties of set-valued
mappings and real-valued functions.

In the next Section~\ref{S2}, we discuss $q$-order ($q>0$) positively
homogeneous mappings and $q$-order graphical (contingent) derivatives. The
definitions and statements mostly follow the corresponding linear ones in
\cite{DonRoc14}. Two norm-like quantities are used for quantifying H\"{o}lder
graphical derivatives. One of them is a generalization of the well-known
\emph{outer norm} of a positively homogeneous mapping, while the other seems
new and allows to simplify some statements (and proofs) even in the linear case.

In Section~\ref{S3}, H\"older graphical derivatives are used for
characterizing H\"older strong subregularity, isolated calmness and sharp
minimum. In particular, we give characterizations of H\"older sharp minimizers
in terms of H\"{o}lder graphical derivatives of the subdifferential mapping.
The characterizations from Section~\ref{S3} are used in Sections~\ref{S4} and
\ref{S5} to characterize H\"older isolated calmness in linear semi-infinite
optimization and sharp minimizers of $\ell_{p}$ penalty functions, respectively.

Our basic notation is standard, see, e.g., \cite{RocWet98,DonRoc14}.
Throughout the paper, $X$ and $Y$ are normed spaces. We use the same notation
$\|\cdot\|$ for norms in all spaces. If not explicitly stated otherwise,
products of normed spaces are assumed equipped with the maximum norms, e.g.,
$\|(x,y)\|:=\max\{\|x\|,\|y\|\}$, {$(x,y)\in X\times Y$.} If $X$ is a normed
space, its topological dual is denoted by $X^{*}$, while $\langle\cdot
,\cdot\rangle$ denotes the bilinear form defining the pairing between the two
spaces. Symbols $\mathbb{R}$, $\mathbb{R}_{+}$ and $\mathbb{N}$ denote the
sets of all real numbers, all nonnegative real numbers and all positive
integers, respectively. For the empty subset of $\mathbb{R}_{+}$, we use the
conventions $\sup\emptyset=0$ and $\inf\emptyset=+\infty$. Given an $\alpha
\in\mathbb{R}$, we denote $\alpha_{+}:=\max\{0,\alpha\}$.

For an extended-real-valued function $f:X\to\mathbb{R}\cup\{+\infty\}$, its
domain and epigraph are defined, respectively, by $\mathrm{dom}\, f:=\{x \in
X\mid{f(x) < +\infty}\}$ and $\mathrm{epi}\, f:=\{(x,\alpha) \in X
\times\mathbb{R}\mid{f(x) \le\alpha}\}$. A set-valued mapping
$F:X\rightrightarrows Y$ between two sets $X$ and $Y$ is a mapping, which
assigns to every $x\in X$ a subset (possibly empty) $F(x)$ of $Y$. We use the
notations $\mathrm{gph}\, F:=\{(x,y)\in X\times Y\mid y\in F(x)\}$ and
$\mathrm{dom}\,\: F:=\{x\in X\mid F(x)\ne\emptyset\}$ for the graph and the
domain of $F$, respectively, and $F^{-1}: Y\rightrightarrows X$ for the
inverse of $F$. This inverse (which always exists with possibly empty values
at some $y$) is defined by $F^{-1}(y) :=\{x\in X \mid y\in F(x)\}$, $y\in Y$.
Obviously $\mathrm{dom}\, F^{-1}=F(X)$.

Recall that a mapping $F:X\rightrightarrows Y$ is \emph{outer semicontinuous}
(cf., e.g., \cite{DonRoc14}) at $x\in X$ if
\[
\limsup_{u\to x}F(u)\subset F(x),
\]
i.e., if $\mathrm{gph}\, F\ni(x_{k},y_{k})\to(x,y)$, then $y\in F(x)$. This is
always the case when $\mathrm{gph}\, F$ is closed.

Throughout the paper, we assume the order of all H\"{o}lder properties to be
determined by a fixed number $q>0$.

\section{H\"{o}lder graphical derivatives}

\label{S2}

In this section, we discuss H\"{o}lder
positively homogeneous mappings
and H\"{o}lder versions of the ubiquitous graphical (contingent) derivatives;
cf. \cite{AubFra90,RocWet98,DonRoc14,Iof17}.

\begin{definition}\label{D1.1}
A mapping $H:X\rightrightarrows Y$ is $q$-order positively homogeneous whenever
\begin{equation*}
0\in H(0)\quad \mbox{and}\quad H(\lambda x)=\lambda^q H(x) \;\;\text{for all}\;\;  x\in X\;\text{and }\lambda>0.
\end{equation*}
\end{definition}
\if{\AK{10/01/21.
Shall we require $H(0)=\{0\}$?}
}\fi

If $q=1$, we simply say the $H$ is positively homogeneous. The graph of a
positively homogeneous mapping is a cone. This is obviously not the case when
$q\ne1$.

The next simple fact is a direct consequence of the definition.

\begin{proposition}
Let $H:X\rightrightarrows Y$ be $q$-order positively homogeneous.
Then $H\iv:Y\rightrightarrows X$ is $\frac{1}{q}$-order positively homogeneous.
\end{proposition}

\if{
\begin{proof}
Obviously $0\in H\iv(0)$.
Let $y\in Y$ and $\la>0$.
Take an arbitrary
$x\in H\iv(y)$.
Then $y\in H(x)$, and consequently, $\la y
\in H(\la^{\frac{1}{q}}x)$.
Thus, $\la^{\frac{1}{q}}x\in H\iv(\la y)$.
It follows that $\la^{\frac{1}{q}} H\iv(y)\subset H\iv(\la y)$.
Since $y\in Y$ and $\la>0$ are arbitrary, the latter inclusion with $\la y$ and $\la\iv$ in place of $y\in Y$ and $\la$, respectively, produces the opposite inclusion $H\iv(\la y) \subset\la^{\frac{1}{q}} H\iv(y)$.
Hence, $H\iv(\la y)=\la^{\frac{1}{q}} H\iv(y)$.
\qed\end{proof}
}\fi

For a $q$-order positively homogeneous mapping $H:X\rightrightarrows Y$, we
define two norm-li\-ke quantities:
\begin{gather}
\label{1.1}\|H\|^{+}_{q}:=\sup\limits_{(x,y)\in\mathrm{gph}\, H\setminus
\{(0,0)\}} \frac{\|y\|}{\|x\|^{q}},\quad\|H\|^{\circleddash}_{q}%
:=\inf\limits_{(x,y)\in\mathrm{gph}\, H\setminus\{(0,0)\}} \frac
{\|y\|}{\|x\|^{q}}.
\end{gather}
When $q=1$,
the first one reduces to the \emph{outer norm} $\|H\|^{+}$ of $H$; cf.
\cite[p.~364]{RocWet98}, \cite[p.~218]{DonRoc14}. Note that
$\|H\|^{\circleddash}_{1}\le\|H\|^{-}$, where $\|H\|^{-}$ is the \emph{inner
norm} of $H$, and the inequality can be strict. None of the quantities in
\eqref{1.1} is actually a true ``norm''; see the comments in \cite{DonRoc14}.

\begin{proposition}\label{P1.3}
Let $H:X\rightrightarrows Y$ be $q$-order positively homogeneous.
\begin{enumerate}
\item
$\|H\|^+_q=\sup_{\|y\|=1}d(0,H\iv(y))^{-q}$,\quad $\|H\|^\circleddash_q=\inf_{\|x\|=1}d(0,H(x))$.
\item
$\|H\|^\circleddash_q=\left(\|H\iv\|^+_\frac{1}{q}\right)^{-q}$.
\item
If $\gph H\ne\{(0,0)\}$, then $\|H\|^\circleddash_q\le\|H\|^+_q$.
\item
If $\gph H=\{(0,0)\}$, then $\|H\|^+_q=0$ and $\|H\|^\circleddash_q=+\infty$.
\item
$\|H\|^+_q=0$ if and only if $H(X)=\{0\}$.
\item
$\|H\|^\circleddash_q=+\infty$ if and only if $\dom H=\{0\}$.
\item
$\|H\|^+_q<+\infty\;\;\Longrightarrow\;\;H(0)=\{0\}$.
If $\dim Y<\infty$ and $H$ is outer semicontinuous at $0$, the two conditions are equivalent.
\item
$\|H\|^\circleddash_q>0\;\;\Longrightarrow\;\;H\iv(0)=\{0\}$.
If $\dim X<\infty$ and $H\iv$ is outer semicontinuous at $0$, the two conditions are equivalent.
\end{enumerate}
\end{proposition}

\if{\AK{14/03/21.
Definition of outer/upper semicontinuity to be given in the Introduction.
Outer semicontinuity is weaker than closedness of the graph.
}}\fi

\begin{proof}
Assertions (i)--(vi) and the first parts of assertions (vii) and (viii) are direct consequences of \eqref{1.1} and Definition~\ref{D1.1}.
For instance, in the case of assertion (ii) using definitions \eqref{1.1} we have:
\begin{gather*}
\|H\|^\circleddash_q =\left(\sup\limits_{(x,y)\in\gph H\setminus\{(0,0)\}} \frac{\|x\|}{\|y\|^\frac{1}{q}}\right)^{-q} =\left(\|H\iv\|^+_\frac{1}{q}\right)^{-q}.
\end{gather*}

To prove the second part of (vii), we need to show that, under the assumptions, $\|H\|^+_q=+\infty\;\;\Longrightarrow\;\;H(0)\ne\{0\}$.
Let $\|H\|^+_q=+\infty$.
By \eqref{1.1}, there exists a sequence $(x_k,y_k)\in\gph H$ $(k\in\N)$ such that $\|y_k\|/\|x_k\|^q\to+\infty$ as $k\to\infty$.
Without loss of generality, $y_k\ne0$ for all $k\in\N$.
Set $u_k:=x_k/\|y_k\|^\frac{1}{q}$, $v_k:=y_k/\|y_k\|$ $(k\in\N)$.
Then $u_k\to0$ as $k\to\infty$ and $\|v_k\|=1$ $(k\in\N)$.
Without loss of generality, $v_k\to v$ as $k\to\infty$ and $\|v\|=1$.
Furthermore, by Definition~\ref{D1.1}, $(u_k,v_k)\in\gph H$ $(k\in\N)$ and, thanks to the outer semicontinuity of $H$, $v\in H(0)$.

The proof of the second part of (viii) is similar.
Let $\dim X<\infty$, $H$ is outer semicontinuous at $0$, and
$\|H\|^\circleddash_q=0$.
By \eqref{1.1}, there exists a sequence $(x_k,y_k)\in\gph H$ $(k\in\N)$ such that $\|y_k\|/\|x_k\|^q\to0$ as $k\to\infty$.
Without loss of generality, $x_k\ne0$ for all $k\in\N$, and
$u_k:=x_k/\|x_k\|\to u$ with $\|u\|=1$, while $v_k:=y_k/\|x_k\|^{q}\to0$.
By Definition~\ref{D1.1}, $(u_k,v_k)\in\gph H$ $(k\in\N)$ and, thanks to the outer semicontinuity of $H\iv$, $u\in H\iv(0)$.
Hence,
$H\iv(0)\ne\{0\}$.
\qed\end{proof}

Assertions (i), (v), (vii) and (viii) in Proposition~\ref{P1.3} generalize and
expand the corresponding parts of \cite[Propositions~4A.6 and 5A.7, and
Exercise~4A.9]{DonRoc14}. The above proof of the second part of (vii) largely
follows that of the corresponding part of \cite[Proposition~4A.6]{DonRoc14}.

Next we briefly consider the case $Y=X^{*}$.

\begin{definition}\label{D2.4}
A mapping $H:X\rightrightarrows X^*$ is \emph{$q$-or\-der positively definite} if there exists a number $\la>0$ such that
$$\ang{x^*,x}\ge\la\|x\|^{q+1}\qdtx{for all}
(x,x^*)\in\gph H.$$
The exact upper bound of all such $\la>0$ is denoted by $\|H\|^*_q$.
\end{definition}

In Definition~\ref{D2.4}, it obviously holds
\begin{align}
\label{1.4}\|H\|^{*}_{q}=\inf\limits_{(x,x^{*})\in\mathrm{gph}\, H,\,x\ne0}
\frac{\left\langle x^{*},x \right\rangle _{+}}{\|x\|^{q+1}}.
\end{align}
In general, the expression in \eqref{1.4} is nonnegative, and the case
$\|H\|^{*}_{q}=0$ means that $H$ is not $q$-or\-der positively definite.

\begin{proposition}\label{P3.12}
Let $H:X\rightrightarrows X^*$.
\begin{enumerate}
\item
$\|H\|^*_q=+\infty$ if and only if $\dom H\subset\{0\}$.
\item
If $H$ is $q$-order positively homogeneous, then $\|H\|^*_q\le\|H\|^\circleddash_q$.
\item
If $H$ is $q$-order positively homogeneous and $p$-order positively definite with some $p>0$, then either $\dom H=\{0\}$ or $p=q$.
\end{enumerate}
\end{proposition}

\begin{proof}
\begin{enumerate}
\item
is immediate from \eqref{1.4}.
\item
follows from comparing \eqref{1.4} and the second definition in \eqref{1.1}.
\item
Let $H$ be $q$-order positively homogeneous and $p$-order positively definite with some $p>0$.
Then $0\in\dom H$, $\|H\|^*_p>0$ and,
by Definition~\ref{D1.1}, $(x,x^*)\in\gph H$ if and only if $(\la x,\la^qx^*)\in\gph H$ for any $\la>0$, and it follows from \eqref{1.4} that
\begin{align*}
\|H\|^*_p=\inf\limits_{(x,x^*)\in\gph H,\,x\ne0,\,\la>0} \frac{\ang{\la^qx^*,\la x}_+}{\|\la x\|^{p+1}} =\inf\limits_{\la>0} \la^{q-p}\|H\|^*_p.
\end{align*}
Thus, either $\|H\|^*_p=+\infty$, i.e. $\dom H=\{0\}$, or $p=q$.
\qed\end{enumerate}
\end{proof}

Given a set-valued mapping $H:X\rightrightarrows Y$ and a function $f:X\to Y$,
their sum $H+f$ is a set-valued mapping from $X$ to $Y$ defined by
\[
(H+f)(x):=H(x)+f(x)=\{y+f(x)\mid y\in H(x)\},\quad x\in X.
\]
Note that $\mathrm{dom}\,(H+f)=\mathrm{dom}\, H\cap\mathrm{dom}\, f$.

The next statement characterizes perturbed positively homogeneous mappings. It
generalizes \cite[Theorem~5A.8]{DonRoc14} (and is accompanied by a much
shorter proof).

\begin{theorem}\label{T1.4}
Let both $H:X\rightrightarrows Y$ and $f:X\to Y$ be $q$-order positively homogeneous.
Then $H+f$ is $q$-order positively homogeneous.
Moreover,
\begin{equation}\label{T1.4-1}
\|H+f\|^\circleddash_q\geq \|H\|^\circleddash_q-\|f\|^+_q.
\end{equation}
\end{theorem}

\begin{proof}
$H+f$ is $q$-order positively homogeneous by Definition~\ref{D1.1}.
If $\dom H\cap\dom f=\{0\}$, then $\|H+f\|^\circleddash_q=+\infty$ by Proposition~\ref{P1.3}(vi), and condition \eqref{T1.4-1} is satisfied trivially.
Let $(x,y)\in\gph H$ with $x\ne0$ and $x\in\dom f$.
Then $(x,y+f(x))\in\gph(H+f)$ and, in view of \eqref{1.1},
\begin{equation*}
\frac{\|y+f(x)\|}{\|x\|^q}\ge\frac{\|y\|}{\|x\|^q} -\frac{\|f(x)\|}{\|x\|^q}\ge\|H\|^\circleddash_q-\|f\|^+_q.
\end{equation*}
Since $(x,y+f(x))$ is an arbitrary point in $\gph(H+f)$ with $x\ne0$, the second representation in \eqref{1.1} yields condition \eqref{T1.4-1}.
\qed\end{proof}

Given a set-valued mapping $F:X\rightrightarrows Y$, its \emph{$q$-order
graphical derivative} at $(\bar x,\bar y)\in\mathrm{gph}\, F$ is a set-valued
mapping $D_{q}F(\bar x, \bar y):X\rightrightarrows Y$ defined for all $x\in X$
by
\begin{align}
\label{2.3}
D_{q} F(\bar x,\bar y)(x):=\big\{y\in Y\mid &  \exists(x_{k}%
,y_{k})\to(x,y),\;t_{k}\downarrow0\;\mbox{such that}\nonumber\\
&  (\bar x+t_{k}x_{k},\bar y+t_{k}^{q}y_{k})\in\mathrm{gph}\, F,\;\forall
k\in\mathbb{N}\big\}.
\end{align}
$D_{q} F(\bar x, \bar y)$ is sometimes referred to as \emph{$q$-order upper
Studniarski derivative} \cite[Definition 3.1]{SunLi11} of $F$ at $(\bar x,\bar
y)$. When $q=1$, it reduces to the standard graphical (contingent) derivative; cf. \cite{AubEke84,AubFra90,RocWet98,KlaKum02,DonRoc14}.
Clearly, $D_{q} F(\bar x, \bar y)$ is a $q$-order positively
homogeneous mapping with closed graph, and
\begin{gather}
D_{q} F(\bar x, \bar y)^{-1}=D_{\frac{1}{q}} F^{-1}(\bar y, \bar x),
\label{4.003}%
\end{gather}

Given a function $f:X\to Y$ and a point $\bar x\in\mathrm{dom}\, f$, we write
$D_{q}f(\bar x)$ instead of $D_{q}f(\bar x,f(\bar x))$. If $D_{q}f(\bar x)$ is
single-valued, i.e. the limit
\[
D_{q}f(\bar x)(x)=\lim_{u\to x,\;t\downarrow0} \frac{f(\bar x+tu)-f(\bar
x)}{t^{q}}%
\]
exists for all $x\in X$, we say that $f$ is \emph{$q$-order Hadamard
directionally differentiable} at~$\bar x$.

The next proposition provides a sum rule for $q$-order graphical derivatives.
It is a direct consequence of the definitions of $q$-order
graphical derivative and $q$-order Hadamard
directional differentiability.

\begin{proposition}
\label{P4.1}
Let $F:X\rightrightarrows Y$, $f:X\to Y$, $(\bx,\by)\in\gph F$ and $\bx\in\dom f$.
If $f$ is $q$-order Hadamard directionally
differentiable at $\bx$, then
\begin{equation*}
D_q(F+f)(\bx,\by+f(\bx)) =D_qF(\bx,\by)+D_qf(\bx).
\end{equation*}
\end{proposition}

\if{
\begin{proof}
Clearly, $(\bx,\by+f(\bx))\in\gph(F+f)$.
Given an $x\in X$ and a
$y\in D_q(F+f)(\bx,\by+f(\bx))(x)$,
there exist sequences $(x_k,y_k)\to(x,y)$ and $t_k\downarrow0$ such that, for all $k\in\mathbb{N}$,
\begin{equation*}
\by+f(\bx)+t_k^q y_k \in F(\bx+t_k x_k)+f(\bx+t_kx_k),
\end{equation*}
or equivalently,
$\by+t_k^q(y_k-v_k)\in F(\bx+t_k x_k)$,
where $v_k:=(f(\bx+t_kx_k)-f(\bx))/t_k^q\to D_qf(\bx)(x)$.
Hence, $y \in  D_q F(\bx,\by)(x)+D_q f (\bx)(x)$.
Conversely,
given an $x\in X$ and a
$y\in D_qF(\bx,\by)(x)$,
there exist sequences $(x_k,y_k)\to(x,y)$ and $t_k\downarrow0$ such that, for all $k\in\mathbb{N}$,
$\by+t_k^q y_k\in F(\bx+t_k x_k)$,
or equivalently,
\begin{equation*}
\by+f(\bx)+t_k^q(y_k+v_k)\in F(\bx+t_k x_k)+f(\bx+t_kx_k),
\end{equation*}
where $v_k:=(f(\bx+t_kx_k)-f(\bx))/t_k^q\to D_qf(\bx)(x)$.
Hence, $y+D_qf(\bx)(x)\in D_q(F+f)(\bx,f(\bx)+\by)(x)$.
\qed\end{proof}
}\fi

Given a function $f:X\to\mathbb{R}\cup\{+\infty\}$, its \emph{$q$-order
Hadamard directional subderivative} \cite{Stu86,StuWar99} at $\bar
x\in\mathrm{dom}\, f$ is defined for all $x\in X$ by (cf.
\cite[Definition~1.1]{Pen78} and \cite[Definition~8.1]{RocWet98} for the case $q=1$)
\begin{equation}
\label{2.5}f^{\prime}_{q}(\bar x;x):=\liminf_{u\to x,\,t\downarrow0}
\frac{f(\bar x+tu)-f(\bar x)}{t^{q}}.
\end{equation}
If $f$ is Lipschitz continuous near $\bar x$ and $0<q\le1$, the above definition takes a
simpler form:
\[
f^{\prime}_{q}(\bar x;x)=\liminf_{t\downarrow0} \frac{f(\bar x+tx)-f(\bar
x)}{t^{q}}.
\]

Observe that the function $f^{\prime}_{q}(\bar x;\cdot):X\to\mathbb{R}%
\cup\{\pm\infty\}$ is lower semicontinuous\ and \emph{$q$-order positively
homogeneous} in the sense that $f^{\prime}_{q}(\bar x;\lambda x)=\lambda^{q}
f^{\prime}_{q}(\bar x;x)$ for all $x\in X$ and $\lambda>0$. We are going to
use for characterizing this function the following norm-li\-ke quantity:
\begin{gather}
\label{2.7}\|f^{\prime}_{q}(\bar x;\cdot)\|_{q}:=\inf\limits_{x\in
X\setminus\{0\}} \frac{(f^{\prime}_{q}(\bar x;x))_{+}}{\|x\|^{q}}%
=\inf\limits_{\|x\|=1} (f^{\prime}_{q}(\bar x;x))_{+}.
\end{gather}

The next statement is a direct consequence of the definitions. It uses the
\emph{epigraphical mapping} $x\mapsto\mathrm{epi}\, f(x):=\{\mu\in
\mathbb{R}\mid f(x)\le\mu\}$. Note that the graph of the latter mapping is the
\emph{epigraph} $\mathrm{epi}\, f$ of $f$.
We use the same notation for the epigraph and the epigraphical mapping.

\begin{proposition}\label{P2.8}
Let $f:X\to\R\cup\{+\infty\}$ and $\bar x\in\dom f$.
\begin{enumerate}
\item
Either $f'_q(\bx;0)=0$ or $f'_q(\bx;0)=-\infty$.
\item
$D_q(\epi f)(\bx,f(\bx))(x)=\{\mu\in\R\mid f'_q(\bx;x)\le\mu\}$ for all $x\in X$.
\item
$\|D_q(\epi f)(\bx,f(\bx))\|_q^+=+\infty$ and $\|D_q(\epi f)(\bx,f(\bx))\|_q^\circleddash=\|f'_q(\bx;\cdot)\|_q$.
\item
$\|f'_q(\bx;\cdot)\|_q>0\;\;\Longrightarrow\;\;f'_q(\bx;x)>0$ for all $x\ne0$.
If $\dim X<\infty$, the two conditions are equivalent.
\end{enumerate}
\end{proposition}

\begin{proof}
\begin{enumerate}
\item
By definition \eqref{2.5}, $f'_q(\bx;0)\le0$.
Suppose that $f'_q(\bx;0)<0$.
Then there exist sequences $u_k\to0$ and $t_k\downarrow0$ such that
\begin{equation*}
\lim_{k\to+\infty} \frac{f(\bx+t_ku_k)-f(\bx)}{t_k^q}=\al<0.
\end{equation*}
For any $\theta>0$ and $k\in\N$, set $u_k':=\theta^{\frac1q}u_k$ and $t_k':=\theta^{-\frac1q}t_k$.
We have $u_k'\to0$ and $t_k'\downarrow0$ and
\begin{equation*}
\lim_{k\to+\infty} \frac{f(\bx+t_k'u_k')-f(\bx)}{(t_k')^q}=\theta\al.
\end{equation*}
Hence, $f'_q(\bx;0)=-\infty$.
\item
The assertion
is immediate from comparing definitions \eqref{2.3} and \eqref{2.5}.
\item
By (i) and (ii), $(0,\mu)\in\gph D_q(\epi f)(\bx,f(\bx))$ for all $\mu\ge0$, and it follows from
\eqref{1.1} that $\|D_q(\epi f)(\bx,f(\bx))\|_q^+=+\infty$.
The second equality holds trivially when $f'_q(\bx;x)=+\infty$ for all $x\ne0$.
If $x\ne0$ and $f'_q(\bx;x)<+\infty$,
then, in view of (ii),
\begin{equation*}
\inf_{\mu\in D_q(\epi f)(\bx,f(\bx))(x)} |\mu|
=\inf_{\mu\ge f'_q(\bx;x)} |\mu|
=(f'_q(\bx;x))_+,
\end{equation*}
and the second equality follows from \eqref{1.1} and \eqref{2.7}.
\item
If $f'_q(\bx;x)\le0$ for some $x\ne0$, then $\|f'_q(\bx;\cdot)\|_q=0$ by definition
\eqref{2.7}.
This proves the implication.
Let $\dim X<\infty$ and $\|f'_q(\bx;\cdot)\|_q=0$.
By \eqref{2.7}, there is a sequence $\{x_k\}$ such that $\|x_k\|=1$ for all $k\in\N$, and $(f'_q(\bx;x_k))_+\to0$ as $k\to\infty$.
Without loss of generality, $x_k\to x$ as $k\to\infty$, $\|x\|=1$ and $(f'_q(\bx;x))_+=0$ since $f'_q(\bx;\cdot)$ is \lsc.
This proves the opposite implication.
\qed\end{enumerate}
\end{proof}

The next corollary is a consequence of Propositions~\ref{P4.1} and \ref{P2.8}.

\begin{corollary} \label{P2.9}
Let $f:X\to\R\cup\{+\infty\}$, $g:X\to\R$ and $\bx\in\dom f$.
If $g$ is $q$-order Hadamard directionally
differentiable at $\bx$, then
\begin{equation*}
(f+g)'_q(\bx;x) =f'_q(\bx;x)+D_qg(\bx)(x)
\quad\text{for all}\quad x\in X.
\end{equation*}
\end{corollary}

\section{H\"older strong subregularity, isolated calmness and sharp minimum}

\label{S3}

In this section, H\"older graphical derivatives are used for characterizing
H\"older strong subregularity, isolated calmness and sharp minimum.

\begin{definition}\label{D3.1}
\begin{enumerate}
\item
A mapping $F:X\rightrightarrows Y$ is $q$-order strongly
subregular at $(\bar{x},\bar{y})\in\gph F$ with modulus $\tau>0$ if there exist neighbourhoods $U$ of $\bar x$ and $V$ of $\bar y$ such that
\begin{equation}\label{D3.1-1}
\tau\|x-\bar x\|^{q}\leq d(\by,F(x)\cap V)\qdtx{for all} x\in U.
\end{equation}
The exact upper bound of all such $\tau>0$ is denoted by $\srgq F(\bar{x},\bar{y})$.
\item
A mapping $S:Y\rightrightarrows X$ possesses $q$-order isolated calmness property at $(\bar{y},\bar{x})\in\gph S$ with modulus $\tau>0$ if there exist neighbourhoods $U$ of $\bar x$ and $V$ of $\bar y$ such that
\begin{equation}\label{D3.1-2}
\tau\|x-\bar x\|^{q}\leq\|y-\bar y\|\qdtx{for all} y\in V\;\mbox{and}\; x\in S(y)\cap U.
\end{equation}
The exact upper bound of all such $\tau>0$ is denoted by $\clmq S(\bar{y},\bar{x})$.
\end{enumerate}
\end{definition}

If $F$ is not $q$-order strongly
subregular at $(\bar{x},\bar{y})$
or $S$ does not possess $q$-or\-der isolated calmness property at $(\bar
{y},\bar{x})$, we have $\mathrm{srg_{q}}\, F(\bar{x},\bar{y})=0$ or
$\mathrm{clm_{q}}\, S(\bar{y},\bar{x})=0$, respectively.

The properties in the above definition are well known in the linear case $q=1$ (see, e.g., \cite{DonRoc14}), but have also been studied in the general setting (also for not necessarily strong subregularity and not necessarily isolated calmness); cf. \cite{Kum09,GayGeoJea11,ChuKim16}.
Because of the distance involved in the \RHS\ of \eqref{D3.1-1} (and also in its \LHS\ in the case of the not strong version), the property in part (i) of Definition~\ref{D3.1} is often referred to as  $q$-order strong \emph{metric} subregularity.

\begin{remark}\label{R3.2}
\begin{enumerate}
\item
In both parts of Definition~\ref{D3.1}, it suffices to take $V:=Y$; cf. \cite[Exercise~3H.4]{DonRoc14}.
\item
Condition \eqref{D3.1-2} implies that $S(\by)\cap U=\{\bx\}$, i.e. $\bx$ is an isolated point in $S(\by)$, which justifies the word `isolated' in the name of the property in Definition~\ref{D3.1}(ii).
\item
The moduli $\srgq F(\bar{x},\bar{y})$ and $\clmq S(\bar{y},\bar{x})$ are usually introduced to characterize the usual (not strong!)
subregularity and (not isolated!) calmness.
We do not consider these two weaker properties in the current paper.
If a respective (strong or isolated) property in Definition~\ref{D3.1} holds, then the corresponding modulus coincides with the conventional one.
\item
When $V=Y$, condition \eqref{D3.1-1} is obviously implied by the following $q$-order strong \emph{graph} subregularity property:
\begin{equation*}
\tau\|x-\bar x\|^{q}\leq d((x,\by),\gph F)\qdtx{for all} x\in U
\end{equation*}
(with the same $\tau$ and $U$).
It is not difficult to show that, when $q\ge1$, $q$-order (strong) subregularity in part (i) of Definition~\ref{D3.1} implies $q$-order (strong) {graph} subregularity
(with smaller $\tau$ and $U$);
cf. a characterization of subregularity in \cite[Proposition 2.61]{Iof17}.
A similar observation can be made about the calmness property in part (ii) of Definition~\ref{D3.1}; cf. the well-known characterization of $q$-order calmness by Kummer \cite[Lemma 2.2]{Kum09}, and the earlier result by Klatte and Kummer \cite[Lemma~3.2]{KlaKum02.2} for the case $q=1$.
\item
There is some inconsistency in the literature concerning whether to place the constants $\tau$ and/or $q$, which determine the properties in Definition~\ref{D3.1}, in the left or \RHS s of the inequalities \eqref{D3.1-1} and \eqref{D3.1-2} (and similar inequalities involved in related definitions); cf., e.g., \cite{KlaKum02}.
This applies also to our own recent paper \cite{KruLopYanZhu19}, where we placed $q$ in the \RHS s of the inequalities.
Of course, the position of the constants does not effect the properties, but it has an effect on the values of the respective moduli.
Our choice in the current paper is determined by our desire to produce the simplest relations between these moduli and the quantitative characteristics of H\"older graphical derivatives and more straightforward proofs.
\end{enumerate}
\end{remark}

The next proposition is an immediate consequence of Definition~\ref{D3.1}.

\begin{proposition}\label{P3.2-}
Let $F:X\rightrightarrows Y$ and $(\bx,\by)\in\gph F$.
Then $F$ is $q$-order strongly
subregular at $(\bar{x},\bar{y})$ if and only if $F\iv$ possesses $q$-or\-der isolated calmness property at $(\bar{y},\bar{x})$, and $\srgq F(\bar{x},\bar{y})=\clmq F\iv(\bar{y},\bar{x})$.
\end{proposition}

The next proposition and its corollaries generalize \cite[Theorem~4E.1 and
Corollary~4E.2]{DonRoc14}.

\if{
\AK{15/02/20.
\cite{DonRoc14} is the 2nd edition.
The reference above is correct.
In the first edition, it was Corollary~4C.2.}
}\fi

\begin{proposition}\label{P3.2}
Let $F:X\rightrightarrows Y$ and $(\bx,\by)\in\gph F$.
Then
\begin{equation}\label{P3.2-1}
\srgq F(\bar{x},\bar{y})\le \|D_{q}F(\bar x,\bar y)\|^\circleddash_{q}.
\end{equation}
If $\dim X<+\infty$ and $\dim Y<+\infty$, then \eqref{P3.2-1} holds as equality.
\end{proposition}

\begin{proof}
If $\srgq F(\bar{x},\bar{y})=0$ or
$\|D_{q}F(\bar x,\bar y)\|^\circleddash_{q}=+\infty$, inequality \eqref{P3.2-1} holds trivially.
Let $\tau\in(0,\srgq F(\bar{x},\bar{y}))$.
Let $u\in X\setminus\{0\}$ and $v\in D_{q}F(\bar x, \bar y)(u)$, i.e. there exist sequences $(u_k,v_k)\to(u,v)$ and $t_k\downarrow0$ such that $\by+t_k^qv_k\in F(\bx+t_ku_k)$ for all $k\in\mathbb{N}$.
By Definition~\ref{D3.1}(i),
$\tau\|u_k\|^{q}\leq\|v_k\|$
for all sufficiently large $k\in\mathbb{N}$, and consequently, $\tau\leq\|v\|/\|u\|^{q}$.
In view of definition \eqref{1.1}, we have
$\tau\leq\|D_{q}F(\bar x,\bar y)\|^\circleddash_{q}$.
Inequality \eqref{P3.2-1} follows.
Let $\dim X<+\infty$, $\dim Y<+\infty$, and $\tau>\srgq F(\bar{x},\bar{y})$.
By Definition~\ref{D3.1}(i),
there exists a sequence $(x_k,y_k)\to(\bx,\by)$
such that $(x_k,y_k)\in\gph F$ and $\tau\|x_k-\bar x\|^{q}>\|y_k-\bar y\|$ for all $k\in\mathbb{N}$.
Then $t_k:=\|x_k-\bx\|\downarrow0$.
Set $u_k:=(x_k-\bx)/t_k$ and $v_k:=(y_k-\by)/t_k^{q}$ ($k\in\N$).
Without loss of generality, $u_k\to u\in X$, $\|u\|=1$, and $v_k\to v\in Y$, $\|v\|\le\tau$.
Thus, $v\in D_{q}F(\bar x,\bar y)(u)$ and
$\|D_{q}F(\bar x,\bar y)\|^\circleddash_{q}\le\|v\|/\|u\|^{q}\leq\tau.$
Hence, \eqref{P3.2-1} holds as equality.
\qed\end{proof}

The following statement provides a characterization of $q$-order strong
subregularity of a mapping in terms of its $q$-order graphical derivative.

\begin{corollary}\label{C1.7+}
Let $F:X\rightrightarrows Y$ and $(\bx,\by)\in\gph F$.
Consider the following conditions:
\begin{enumerate}
\item
$F$ is $q$-order strongly subregular at $(\bx,\by)$;
\item
$\|D_{q}F(\bx,\by)\|^\circleddash_{q}>0$;
\item
$D_{q}F(\bx,\by)\iv(0)=\{0\}$.
\end{enumerate}
Then {\rm (i) $\Rightarrow$ (ii) $\Rightarrow$ (iii)}.
If $\dim X<+\infty$ and $\dim Y<+\infty$, then {\rm (i) $\Leftrightarrow$ (ii) $\Leftrightarrow$ (iii)}.
\end{corollary}

\begin{proof}
Thanks to Proposition~\ref{P3.2}, we have the implication (i) $\Rightarrow$ (ii) in general, and the equivalence (i) $\Leftrightarrow$ (ii) when $\dim X<+\infty$ and $\dim Y<+\infty$.
The implication (ii) $\Rightarrow$ (iii) is an
immediate consequence of Proposition~\ref{P1.3}(viii).
The graph of $D_{q}F(\bx,\by)$ is closed by definition, hence, $D_{q}F(\bx,\by)\iv$ is outer semicontinuous at $0$.
Employing Proposition~\ref{P1.3}(viii) again, we conclude that (ii) $\Leftrightarrow$ (iii) when $\dim X<+\infty$.
\qed\end{proof}

\begin{remark}
A coderivative analogue (employing a special kind of limiting coderivative) of
the equality in Corollary~\ref{C1.7+}(iii)
is used in \cite[Theorem 5.2]{ZheZhu16} to characterize nonlinear subregularity.
\end{remark}

The next example illustrates application of Corollary~\ref{C1.7+} for checking H\"older strong subregularity as well as computation of the H\"older graphical derivative and relevant norm-like quantity.

\begin{example}
Let $F:\R\rightrightarrows\R$ be the epigraphical mapping: 
$F(x):=[x^2,+\infty)$
for all $x\in\R$, and $\bx=\by=0$.
By definition~\eqref{2.3}, $y\in D_{q}F(0,0)(x)$ if and only if there exist sequences $(x_{k},y_{k})\to(x,y)$ and $t_{k}\downarrow0$ such that
$y_{k}\ge t_{k}^{2-q}x_{k}^2$ for all
$k\in\mathbb{N}$, or equivalently, $y\ge \lim_{k\to\infty}t_{k}^{2-q}x^2$.
Thus, there are three distinct possibilities.

\underline{$0<q<2$}.
$D_{q}F(0,0)(x)=\R_+$ for all $x\in\R$.
Thus, $D_{q}F(0,0)\iv(0)=\R$ and, by \eqref{1.1}, $\|D_{q}F(0,0)\|^\circleddash_{q}=0$.
Each of the conditions (ii) and (iii) in Corollary~\ref{C1.7+} yields that $F$ is not $q$-order strongly subregular at $(0,0)$.

\underline{$q=2$}.
$D_{q}F(0,0)(x)=[x^2,+\infty)$ for all $x\in\R$.
Thus, $D_{q}F(0,0)\iv(0)=\{0\}$ and, by \eqref{1.1}, $\|D_{q}F(0,0)\|^\circleddash_{q}=1$.
Each of the conditions (ii) and (iii) in Corollary~\ref{C1.7+} yields that $F$ is $q$-order strongly subregular at $(0,0)$.

\underline{$q>2$}.
$D_{q}F(0,0)(0)=\{0\}$, and $D_{q}F(0,0)(x)=\es$ for all $x\ne0$.
Thus, $D_{q}F(0,0)\iv(0)=\{0\}$ and, by \eqref{1.1}, $\|D_{q}F(0,0)\|^\circleddash_{q}=+\infty$.
Each of the conditions (ii) and (iii) in Corollary~\ref{C1.7+} yields that $F$ is $q$-order strongly subregular at $(0,0)$.
\sloppy

Of course, in this simple example, the same conclusions can be obtained directly from Definition~\ref{D3.1}(i).
\end{example}

\begin{corollary}\label{P3.1}
Let $S:Y\rightrightarrows X$ and $(\by,\bx)\in\gph S$.
Then
\begin{equation}\label{P3.1-1}
\clmq S(\bar{y},\bar{x})\le\left(\|D_\frac{1}{q}S(\bar y,\bar x)\|^+_\frac{1}{q}\right)^{-q}.
\end{equation}
If $\dim X<+\infty$ and $\dim Y<+\infty$, then \eqref{P3.1-1} holds as equality.
\end{corollary}

\begin{proof}
By Proposition~\ref{P1.3}(ii) and \eqref{4.003}, we have
\begin{equation*}
\|D_{q}F(\bar x,\bar y)\|^\circleddash_q
=\left(\|D_{q}F(\bar x,\bar y)\iv\|^+_\frac{1}{q}\right)^{-q}
=\left(\|D_\frac{1}{q}F\iv(\bar y,\bar x)\|^+_\frac{1}{q}\right)^{-q}.
\end{equation*}
The assertion follows from Propositions~\ref{P3.2-} and \ref{P3.2}.
\qed\end{proof}

\begin{corollary}\label{T1.6}
Let $S:Y\rightrightarrows X$ and $(\by,\bx)\in\gph S$.
Consider the following conditions:
\begin{enumerate}
\item
$S$ possesses $q$-order isolated calmness property at $(\bar y, \bar x)$;
\item
$\|D_{\frac{1}{q}} S(\bar y,\bar x)\|^+_\frac{1}{q} <+\infty$;
\item
$D_{\frac{1}{q}}S(\bar y, \bar x)(0)=\{0\}$.
\end{enumerate}
Then {\rm (i) $\Rightarrow$ (ii) $\Rightarrow$ (iii)}.
If $\dim X<+\infty$ and $\dim Y<+\infty$, then {\rm (i) $\Leftrightarrow$ (ii) $\Leftrightarrow$ (iii)}.
\end{corollary}

\begin{proof}
Thanks to Corollary~\ref{P3.1}, we have the implication (i) $\Rightarrow$ (ii) in general, and the equivalence (i) $\Leftrightarrow$ (ii) when $\dim X<+\infty$ and $\dim Y<+\infty$.
The implication (ii) $\Rightarrow$ (iii) is an
immediate consequence of Proposition~\ref{P1.3}(vii).
The graph of $D_{\frac{1}{q}}S(\bar y, \bar x)$ is closed by definition, hence, $D_{\frac{1}{q}}S(\bar y, \bar x)$ is outer semicontinuous at $0$.
Employing Proposition~\ref{P1.3}(vii) again, we conclude that (ii) $\Leftrightarrow$ (iii) when $\dim Y<+\infty$.
\qed\end{proof}

\if{
\AK{7/08/19.
I have waived the assumptions that $q\le1$ and $\gph S$ is (locally) closed:
I don't see where they are needed.
The closedness assumption is present in \cite{DonRoc14}.
Is it needed there?}
}\fi

The following theorem shows that $q$-order strong subregularity enjoys
stability under perturbations by functions with small ${q}$-order Hadamard
directional derivatives.

\begin{theorem}
Let $\dim X<+\infty$, $\dim Y<+\infty$,  $F:X\rightrightarrows Y$, $g:X\to Y$, $(\bx,\by)\in\gph F$, $\bx\in\dom g$, and
$g$ be ${q}$-order Hadamard directionally differentiable at $\bx$.
Then
\begin{equation*}
\srgq(F+g)(\bx,\by+g(\bx))\geq \srgq F(\bar{x},\bar{y}) -\|D_{q}g(\bx)\|^+_{q}.
\end{equation*}
If $\|D_{q}F(\bx,\by)\|^\circleddash_{q} >\|D_{q}g(\bx)\|^+_{q}$, then $F+g$ is $q$-order strongly subregular at $(\bx,\by+g(\bx))$.
\end{theorem}

\begin{proof}
By Proposition \ref{P4.1} and Theorem \ref{T1.4},
\begin{align*}\notag
\|D_{q}(F+g)(\bx,\by+g(\bx))\|^\circleddash_{q}
&=\|D_{q}F(\bx,\by) +D_{q}g(\bx)\|^\circleddash_{q}
\\\notag
&\geq \|D_{q}F(\bx,\by)\|^\circleddash_{q} -\|D_{q}g(\bx)\|^+_{q}.
\end{align*}
The assertion follows from Proposition~\ref{P3.2} and Corollary~\ref{C1.7+}.
\qed\end{proof}

The next proposition is a consequence of Proposition~\ref{P3.12}(ii) and
Proposition~\ref{P3.2}.

\begin{proposition}\label{P3.6}
Let $\dim X<+\infty$, $F:X\rightrightarrows X^*$ and $(\bx,x^*)\in\gph F$.
Then $\|D_{q}F(\bx,x^*)\|^*_{q}\le\srgq F(\bx,x^*)$.
As a consequence,
if $D_{q}F(\bx,x^*)$ is $q$-order positively definite with modulus $\lambda>0$,
then
$F$ is $q$-order strongly subregular at $(\bar x,x^*)$
with any modulus $\tau\in\left(0,\lambda\right)$.
\end{proposition}

Given a function $f:X\to\mathbb{R}\cup\{+\infty\}$, it is easy to check
(taking into account Remark~\ref{R3.2}(i)) that the $q$-order strong
subregularity of its epigraphical mapping
at $(\bar x,f(\bar x))$ reduces to
the property in the next definition.

\begin{definition}\label{D3.11}
Let $f:X\to\R\cup\{+\infty\}$.
A point $\bar x\in\dom f$ is a $q$-order sharp minimizer of $f$
with modulus $\tau>0$ if there exists a neighbourhood $U$ of $\bar x$ such that
\begin{equation}\label{D3.11-1}
\tau\|x-\bar x\|^q\leq f(x)-f(\bar x)\qdtx{for all} x\in U.
\end{equation}
The exact upper bound of all such $\tau>0$ is denoted by $\srpq f(\bar{x})$.
\end{definition}

If $\bar{x}$ is not a $q$-order sharp minimizer of $f$, we have
$\mathrm{shrp_{q}}\, f(\bar{x})=0$.

\begin{remark}
\begin{enumerate}
\item
The property in Definition~\ref{D3.11} is also known as \emph{isolated local minimum with order $q$}; cf. \cite{Stu86}.
\item
If $f(\bar{x})=0$ and $\bar{x}$ is a $q$-order sharp minimizer of $f$, then $\srpq f(\bar{x})$ coincides with the $q$-order error bound modulus $\Erq f(\bx)$ of $f$ at $\bar{x}$.
\end{enumerate}
\end{remark}

The next proposition is a consequence of Proposition~\ref{P3.2} and
Proposition~\ref{P2.8}(iii).

\begin{proposition}\label{P3.13}
Let $f:X\to\R\cup\{+\infty\}$ and
$\bar x\in\dom f$. Then
\begin{equation}\label{P3.13-1}
\srpq f(\bar{x})\le\|f'_q(\bx;\cdot)\|_q.
\end{equation}
If $\dim X<+\infty$, then \eqref{P3.13-1} holds as equality.
\end{proposition}

The following lemma describing H\"{o}lder sharp minimizers in terms of the
H\"{o}lder strong subregularity of the subdifferential mappings is a
reformulation of \cite[Theorem 4.1]{ZheNg15} in the convex setting.

\begin{lemma}
\label{T4.3}
Let $X$ be a Banach space, $f:X\to\mathbb{R}\cup\{+\infty\}$
lower semicontinuous and convex, and $\bar x\in\dom f$ be a local minimizer of $f$.
Consider the following assertions:
\begin{enumerate}
\item
$\bar x$ is a $(q+1)$-order sharp minimizer of $f$
with modulus $\rho>0$;
\item
$\partial f$ is $q$-order strongly subregular at $(\bar x, 0)$
with modulus $\tau>0$.
\end{enumerate}
Then {\rm (i) $\Rightarrow$ (ii)} with $\tau:=\rho$, and
{\rm (ii) $\Rightarrow$ (i)} with $\rho:=\frac{q^{q}}{(q+1)^{q+1}}\tau
$.
As a consequence,
$\bar x$ is a $(q+1)$-order sharp minimizer of $f$ if and only if
$\partial f$ is $q$-order strongly subregular at $(\bar x, 0)$.
\end{lemma}

Next we give
characterizations of H\"older sharp minimizers in terms of H\"{o}lder
graphical derivatives of the subdifferential mapping. The theorem below is
partially motivated by \cite[Corollary 3.7]{AraGeo14}, which provides a
characterization of the strong subregularity in terms of the
positive-definiteness of the graphical derivative. The modulus estimate in the
following theorem is inspired by \cite[Theorem 3.6]{MorNgh15}, where a
characterization of tilt stability of local minimizers for
extended-real-valued functions is derived via the second-order subdifferential.

\begin{theorem}
\label{T1.12}
Let $X$ be a Banach space, $f:X\to\mathbb{R}\cup\{+\infty\}$
lower semicontinuous and convex, and $\bar x\in\dom f$ be a local minimizer of $f$.
Consider the following assertions:
\begin{enumerate}
\item
$\bar x$ is a $(q+1)$-order sharp minimizer of $f$
with modulus $\rho>0$;
\item
$D_q\partial f(\bar x, 0)$ is $q$-order positively definite with modulus $\lambda>0$.
\end{enumerate}
Then {\rm (i) $\Rightarrow$ (ii)} with $\la:=\rho$.
If $\dim X<+\infty$, then, for any $\tau\in\left(0,\lambda\right)$, {\rm (ii) $\Rightarrow$ (i)} with $\rho:=\frac{q^{q}}{(q+1)^{q+1}}\tau$.
As a consequence, if $\dim X<+\infty$, then
$\bar x$ is a $(q+1)$-order sharp minimizer of $f$ if and only if
$D_q\partial f(\bar x, 0)$ is $q$-order positively definite, and
\begin{equation}
\label{T1.12-1}
\frac{q^{q}}{(q+1)^{q+1}}\|D_q\partial f(\bx,0)\|^*_q \le{\rm shrp_{q+1}}\, f(\bx)\le \|D_q\partial f(\bx,0)\|^*_q.
\end{equation}
\end{theorem}

\begin{proof}
Let (i) hold.
Let $u\in X$ and $u^*\in D_q\partial f(\bar x, 0)(u)$, i.e.
there are sequences $t_k\downarrow0$ and $(u_k,u_k^*)\to(u,u^*)$ such that $t_k^qu_k^*\in\partial f(\bx+t_ku_k)$ for all $k\in\mathbb{N}$.
For all sufficiently large $k$, we have $\rho(t_k\|u_k\|)^{q+1}\leq f(\bx+t_ku_k)-f(\bx)\le \langle t_k^{q}u_k^*,t_k u_k\rangle$, and consequently,
$\rho\|u\|^{q+1}\leq \langle u^*,u\rangle$, i.e. $D_q\partial f(\bar x, 0)$ is $q$-order positively definite with modulus $\lambda:=\rho$.
Let (ii) hold,
$\dim X<+\infty$, and $\tau\in\left(0,\lambda\right)$.
By Proposition~\ref{P3.6}, $\partial f$ is $q$-order strongly subregular at $(\bar x, 0)$
with modulus $\tau$.
By Lemma~\ref{T4.3}, $\bar x$ is a $(q+1)$-or\-der sharp minimizer of $f$
with modulus $\frac{q^{q}}{(q+1)^{q+1}}\tau$.
\qed\end{proof}

\if{
\AK{10/08/19.
I have waived the assumptions that $X$ is Hilbert and $q\le\frac{1}{2}$:
I don't see where they are needed.
The convexity assumptions in Lemma \ref{T4.3} and Theorem~\ref{T1.12} can be weakened.
In \cite{ZheNg15}, they seem to use Clarke subdifferentials.
---
JZ (16/12/19) thinks that the convexity assumptions cannot be weakened.}
\AK{26/01/21.
Implication {\rm (ii) $\Rightarrow$ (i)} in \cite[Theorem 4.1]{ZheNg15} (Lemma \ref{T4.3} above) is formulated without convexity.}
}\fi

The next example illustrates application of Theorem~\ref{T1.12} for checking the order of sharp minimizers.

\begin{example}
Let $f(x)=x^{2n}$ for some integer $n>0$ and all $x\in\R$, and $\bx=0$.
Thus, $f'(x)=x^{2n-1}$ for all $x\in\R$ and,
by definition~\eqref{2.3}, $y\in D_{q}\sd f(0,0)(x)$ if and only if there exist sequences $(x_{k},y_{k})\to(x,y)$ and $t_{k}\downarrow0$ such that
$y_{k}=t_{k}^{2n-1-q}x_{k}^{2n-1}$ for all
$k\in\mathbb{N}$, or equivalently, $y=\lim_{k\to\infty}t_{k}^{2n-1-q}x^{2n-1}$.
Thus, there are three distinct possibilities.

\underline{$0<q<2n-1$}.
$D_{q}\sd f(0,0)(x)=\{0\}$ for all $x\in\R$.
Thus, $yx=0$ for all $(x,y)\in\gph D_{q}\sd f(0,0)$, i.e. $D_{q}\sd f(0,0)$ is not $q$-order positively definite.
By Theorem~\ref{T1.12}, $0$ is not a $(q+1)$-order sharp minimizer of $f$.

\underline{$q=2n-1$}.
$D_{q}\sd f(0,0)(x)=\{x^{2n-1}\}$ for all $x\in\R$.
Thus, $yx=x^{2n}=|x|^{2n}$ for all $(x,y)\in\gph D_{q}\sd f(0,0)$, i.e. $D_{q}\sd f(0,0)$ is $(2n-1)$-order positively definite with any modulus $\la\in(0,1]$, and $\|D_q\partial f(0,0)\|^*_q=1$.
By Theorem~\ref{T1.12}, $0$ is a $2n$-order sharp minimizer of $f$, and $\frac{(2n-1)^{2n-1}}{(2n)^{2n}}\le{\rm shrp_{2n}}\, f(0)\le1$.

\underline{$q>2n-1$}.
$D_{q}\sd f(0,0)(0)=\{0\}$, and $D_{q}\sd f(0,0)(x)=\es$ for all $x\ne0$.
Thus, $yx=0$ for all $(x,y)\in\gph D_{q}\sd f(0,0)$, i.e. $D_{q}\sd f(0,0)$ is not $q$-order positively definite.
By Theorem~\ref{T1.12}, $0$ is not a $(q+1)$-order sharp minimizer of $f$.
\sloppy

Of course, in this simple example, the same conclusions can be obtained directly from Definition~\ref{D3.11}.
Moreover, ${\rm shrp_{2n}}\, f(0)=1$, i.e. the lower estimate in \eqref{T1.12-1} is not sharp.
\end{example}

Comparing the statements of Proposition~\ref{P3.6}, Lemma \ref{T4.3} and
Theorem~\ref{T1.12}, we arrive at the following corollary, which provides an
important special case when the implication in Proposition~\ref{P3.6} holds as equivalence.

\begin{corollary}
Let $\dim X<+\infty$, $f:X\to\mathbb{R}\cup\{+\infty\}$ be
lower semicontinuous and convex, and $\bar x\in\dom f$ be a local minimizer of $f$.
Then $D_q\partial f(\bar x, 0)$ is $q$-order positively definite if and only if
$\partial f$ is $q$-order strongly subregular at $(\bar x, 0)$, and
\begin{equation*}
\|D_q\partial f(\bx,0)\|^*_q \le\srgq\sd f(\bx,0)\le \frac{(q+1)^{q+1}}{q^{q}}\|D_q\partial f(\bx,0)\|^*_q.
\end{equation*}
\end{corollary}

\section{$q$-order isolated calmness in linear semi-infinite optimization}

\label{S4}

In this section, we consider a canonically perturbed linear semi-infinite
optimization problem:
\[
\begin{aligned} P(c,b):\quad& \text{minimize} &&\langle c,x\rangle\\ &\text{subject to} &&\left\langle a_t,x \right\rangle\leq b_t,\;t\in T, \end{aligned}
\]
where $x\in\mathbb{R}^{n}$ is the vector of variables, $c\in\mathbb{R}^{n}$,
$\langle\cdot,\cdot\rangle$ represents the usual inner product in
$\mathbb{R}^{n}$, $T$ is a compact Hausdorff space, and the function
$t\mapsto(a_{t},b_{t})$ is continuous on $T$. In this setting, the pair
$(c,b)\in\mathbb{R}^{n}\times C(T,\mathbb{R})$ is regarded as the perturbation
parameter. The parameter space $\mathbb{R}^{n}\times C(T,\mathbb{R})$ is
endowed with the uniform convergence topology through the maximum norm
$\|(c,b)\|:=\max\{\|c\|,\|b\|_{\infty}\},$ where $\|\cdot\|$ is the Euclidean
norm in $\mathbb{R}^{n}$ and $\|b\|_{\infty}:=\max_{t\in T}|b_{t}|$.

The \emph{feasible set} and \emph{solution} mappings corresponding to the
above problem are
defined, respectively, by
\begin{gather}
\label{F}
\mathcal{F}(b):=\{x\in\mathbb{R}^{n}\mid\left\langle a_{t},x
\right\rangle \leq b_{t},\;t\in T\},\quad b\in C(T,\mathbb{R}),\\
\label{S}
\mathcal{S}(c,b):=\{x\in\mathcal{F}(b)\mid x\;\;\mathrm{solves}%
\;\;P(c,b)\},\quad(c,b)\in\mathbb{R}^{n}\times C(T,\mathbb{R}).
\end{gather}

From now on, we assume a point $((\bar{c},\bar{b}),\bar{x})\in\mathrm{gph}%
\,\mathcal{S}$ to be given. We are going to consider also the partial solution
mapping $\mathcal{S}_{\bar c}:C(T,\mathbb{R})\rightrightarrows\mathbb{R}^{n}$
given by $\mathcal{S}_{\bar c}(b)=\mathcal{S}(\bar c,b)$ and the \textit{level
set} mapping
\begin{gather*}
\mathcal{L}(\alpha,b):=\{x\in\mathcal{F}(b)\mid\langle\bar{c},x\rangle
\leq\alpha\}, \quad(\alpha,b)\in\mathbb{R}\times C(T,\mathbb{R}),
\end{gather*}
and employ the following convex and continuous function:
\begin{gather}
\label{barf}{f}(x):=\max\{\langle\bar{c},x-\bar{x}\rangle,\;\max_{t\in
T}(\left\langle a_{t},x \right\rangle -\bar{b}_{t})\},\quad x\in\mathbb{R}%
^{n}.
\end{gather}
Observe that $f(\bar x)=0$, and
\begin{gather}
\label{marco6}\mathcal{S}(\bar{c},\bar{b})=\left[  {f}=0\right]  =\left[  {f}
\leq0\right]  =\mathcal{L}\langle\bar{c},\bar{x}\rangle,\bar{b}).
\end{gather}

The problem $P(c,b)$ satisfies the \textit{Slater condition} if there exists a
point $\hat{x}\in\mathbb{R}^{n}$ such that $\left\langle a_{t},\hat{x}
\right\rangle <b_{t}$ for all $t\in T$. The set of \textit{active indices} at
$x\in\mathcal{F}(b)$ is defined by $T_{b}(x):=\{{t\in T}\mid\left\langle
a_{t},{x} \right\rangle =b_{t}\}.$

The following lemma is an analogue of \cite[Proposition 4.5]{KruLopYanZhu19}.

\begin{lemma}
\label{marco7} Suppose that $P(\bar{c},\bar{b})$ satisfies the Slater
condition, and $\mathcal{S}(\bar{c},\bar{b})=\{\bar{x}\}$. If $\mathcal{L}$
does not possess $q$-order isolated calmness property at $((\langle\bar
{c},\bar{x}\rangle,\bar{b}),\bar{x})\in\mathrm{gph}\,(\mathcal{L})$, then
there exist a sequence $\{(b_{k},x_{k})\}\subset\mathrm{gph}\,\mathcal{F}$
such that $x_{k}\neq\bar{x}$ for all $k\in\mathbb{N}$, and
\[
\lim_{k\rightarrow+\infty}(b_{k},x_{k})=(\bar{b},\bar{x}),\quad\lim
_{k\rightarrow+\infty}\frac{\Vert b_{k}-\bar{b}\Vert_{\infty}}{\left\Vert
x_{k}-\bar{x}\right\Vert ^{q}}=0,
\]
a finite subset $T_{0}\subset\cap_{k\in\mathbb{N}}T_{b_{k}}(x_{k})$, and
positive scalars $\gamma_{t},\ t\in T_{0}$, satisfying
\begin{equation}
-\bar{c}\in\sum_{t\in T_{0}}\gamma_{t}a_{t}.\label{3.17}%
\end{equation}
\end{lemma}

\begin{theorem}\label{T4.2}
Suppose that $P(\bar{c},\bar{b})$ satisfies the Slater condition.
Consider
the following assertions:
\begin{enumerate}
\item
$\mathcal{S}$ possesses $q$-order isolated calmness property at $((\bar{c},\bar{b}),\bar{x})$;
\item
$\mathcal{S}_{\bar{c}}$ possesses $q$-order isolated calmness property at $(\bar{b},\bar{x})$;
\item
$\mathcal{L}$ possesses $q$-order isolated calmness property at $((\langle \bar{c},\bar{x}\rangle,\bar{b}),\bar{x})$;
\item
$\bar x$ is a $q$-order sharp minimizer of $f$;
\item
$\|f'_q(\bx;\cdot)\|_q>0$;
\item
$\|D_{\frac{1}{q}} S_{\bar{c}}(\bar{b},\bar{x})\|^+_\frac{1}{q} <+\infty$.
\cnta
\end{enumerate}
Then {\rm (i) $\Leftrightarrow$ (ii) $\Leftrightarrow$ (iii) $\Leftrightarrow$ (iv) $\Leftrightarrow$ (v) $\Rightarrow$ (vi)}.
{If} $T$ is finite, then all the assertions are equivalent,
{and}
\begin{equation}\label{T4.2-1}
\clmq\,\mathcal{S}_{\bar{c}}(\bar{b},\bar{x}%
)=\left(\left\Vert D_{\frac{1}{q}}\mathcal{S}_{\bar{c}}(\bar{b},\bar{x}%
)\right\Vert^+_\frac{1}{q}\right)^{-q}.
\end{equation}
If $q>1$, then assertions {\rm (i)--(v)} are equivalent to the next one:
\begin{enumerate}
\cntb
\item
$D_{q-1}\partial f(\bar{x},0)$ is $(q-1)$-order
positively definite.
\end{enumerate}
\end{theorem}

\begin{proof}
(i) $\Rightarrow$ (ii) is immediate from Definition~\ref{D3.1}(ii) in view of
the definition of $\mathcal{S}_{\bar{c}}$.

(ii) $\Rightarrow$ (iii). Suppose that $\mathcal{L}$ does not possess
$q$-order isolated calmness property at $((\langle\bar{c},\bar{x}\rangle
,\bar{b}),\bar{x})$. To reach a contradiction with (ii), it suffices to show
that, for the sequence $\{(b_{k},x_{k})\}\subset\mathrm{gph}\,\mathcal{F}$ in
Lemma \ref{marco7}, it holds $x_{k}\in\mathcal{S}_{\bar{c}}(b_{k})$,
$k\in\mathbb{N}$, which readily follows from the KKT conditions (\ref{3.17})
(by continuity, it is not restrictive to assume that $P(\bar{c},b_{k})$
satisfies the Slater condition).

(iii) $\Leftrightarrow$ (iv) follows from comparing Definition~\ref{D3.1}(ii)
and Definition~\ref{D3.11} in view of \eqref{barf} and \eqref{marco6}.

(iv) $\Rightarrow$ (i). By \cite[Lemma 4.2]{KruLopYanZhu19}
(with $f\equiv0$),
there exist a number $M>0$ and neighbourhoods $U$ of
$\bar{x}$ and $V$ of $(\bar{c},\bar{b})$ such that
\begin{equation}
-c\in\lbrack0,M]\ \mathrm{co}\left\{  a_{t},\ t\in T_{b}\left(  x\right)
\right\}  \quad\text{for all }(c,b)\in V\text{ and }x\in\mathcal{S}(c,b)\cap
U,\label{T4.1P1}%
\end{equation}
where `co' stands for the convex hull. Let $\bar{x}$ be a $q$-order sharp
minimizer of $f$. By Definition~\ref{D3.11}, condition \eqref{D3.11-1} holds
with some number $\tau>0$ and a smaller neighbourhood $U$ if necessary.
Without loss of generality, we assume that $M>1$, and $U$ is bounded: $\Vert
x-\bar{x}\Vert<\delta$ for some $\delta>0$ and all $x\in U$. Let $(c,b)\in V$
and $x\in\mathcal{S}(c,b)\cap U$. By \eqref{T4.1P1},
\begin{equation}
-c=\sum_{t\in T_{b}\left(  x\right)  }\eta_{t}a_{t},\label{T4.1P2}%
\end{equation}
for some $\eta_{t}\geq0,\ t\in T_{b}\left(  x\right)  ,$ satisfying
$\sum_{t\in T_{b}\left(  x\right)  }\eta_{t}\leq M$ and only finitely many
being positive. Hence, in view of representation \eqref{T4.1P2}, and
definitions \eqref{F} and \eqref{S},
\[
\langle c,x-\bar{x}\rangle=-\sum_{t\in T_{b}\left(  x\right)  }\eta
_{t}\left\langle a_{t},x-\bar{x}\right\rangle \leq\sum_{t\in T_{b}\left(
x\right)  }\eta_{t}(\bar{b}_{t}-b_{t})\leq M\Vert b-\bar{b}\Vert_{\infty},
\]
Recalling definition \eqref{barf} and the fact that $f(\bar{x})=0$, we have
\begin{align*}
\tau\left\Vert x-\bar{x}\right\Vert ^{q}\leq f(x) &  \leq\max\{\langle\bar
{c},x-\bar{x}\rangle,\;\max_{t\in T}(b_{t}-\bar{b}_{t})\}\\
&  \leq\max\{\langle c,x-\bar{x}\rangle+\Vert c-\bar{c}\Vert\Vert x-\bar
{x}\Vert,\;\Vert b-\bar{b}\Vert_{\infty}\}\\
&  \leq\max\{M\Vert b-\bar{b}\Vert_{\infty}+\delta\Vert c-\bar{c}\Vert,\;\Vert
b-\bar{b}\Vert_{\infty}\}\\
&  \leq(M+\delta)\Vert(c,b)-(\bar{c},\bar{b})\Vert.
\end{align*}
By Definition~\ref{D3.1}(ii), $\mathcal{S}$ possesses $q$-order isolated
calmness property at $((\bar{c},\bar{b}),\bar{x})$. (iv) $\Leftrightarrow$ (v)
is immediate from Proposition~\ref{P3.13}. (ii) $\Rightarrow$ (vi) and the
opposite implication when $T$ is finite, together with the equality
\eqref{T4.2-1} follow from Corollaries~\ref{P3.1} and \ref{T1.6}. It suffices
to notice that, when $T$ is finite, the parameter space $C(T,\mathbb{R})$ is
finite-dimensional. When $q>1$, the equivalence (iv) $\Leftrightarrow$ (vii)
is a consequence of Theorem~\ref{T1.12}. \qed
\end{proof}

\begin{remark}
Implication (iii) $\Rightarrow$ (i) in Theorem~\ref{T4.2} is a consequence of
\cite[Corollary 3]{KlaKum} and the fact that, under the Slater condition, $\mathcal{F}$ is calm and Lipschitz \lsc.

In the case $q\geq1$, implication (iv) $\Rightarrow$ (i) is
a special case of \cite[Theorem 2.2]{Kla94}.
For the semi-infinite
optimization model $P(c,b),$ this implication was explicitly given, e.g., in
\cite[Proposition 4.2]{KlaHen98}.
Indeed, $\bx$ is a
$q$-order sharp minimizer of $f$, then, using the notation of Definition~\ref{D3.1},
one has in particular%
\[
\tau\left\Vert x-\bx\right\Vert ^{q}\leq\left\langle c,x-\bx\right\rangle ,\ \text{\ for all }x\in\mathcal{F}(\bar{b})\cap U,
\]
i.e.,  $\bx$ is a strict local minimizer of $P(\bar{c},\bar{b})$ in the sense of \cite{KlaHen98}. Since the Slater
condition is equivalent to the extended Mangasarian-Fromovitz CQ
(for
this equivalence in relation to the linear SIP problem $P(c,b)$ see, e.g.,
\cite[Theorem~6.1]{GobLop98.2} and \cite[Theorem~2.1]{CanDonLopPar05}), \cite[Proposition 4.2]{KlaHen98} applies and
gives (in particular) that $\mathcal{S}$ possesses the $q$-order isolated
calmness property at $((\bar{c},\bar{b}),\bx)$.
\end{remark}

\if{
\textbf{Additional References}

1) Klatte, D.: On quantitative stability for non-isolated minima. Control and
Cybernetics 23, 183-200 (1994).

2) Reemtsen, R., R\"{u}ckmann, J.-J., eds.: Semi-Infinite Programming.
Nonconvex Optimization and Its Applications, Vol. 25, Kluwer Academic Publ.,
Boston, Dordrecht, London

3) Klatte, D., Henrion, R.: Regularity and stability in nonlinear semiinfinite
optimization. In: Reemtsen-R\"{u}ckmann (1998), pages 69-102.

4) Goberna-L\'{o}pez (1998)
\'{}
Goberna, M.A., L%
\'{}%
opez, M.A.: A comprehensive survey of linear semiinfinite optimization theory.
In Reemtsen-R\"{u}ckmann (1998), pages 3-27.

5) C\'{a}novas, M.J., Dontchev, A.L., L\'{o}pez, M.A., Parra, J.: Metric
regularity of semi-infinite constraint systems. Math. Program., Ser. B 104,
329-346 (2005).
}\fi

Next we recall the
\emph{Extended N\"{u}rnberger Condition }(\emph{ENC, }in brief)
\cite[Definition~2.1]{CanHanLopPar08}.

\begin{definition}
ENC is satisfied at $((\bar{c},\bar{b}),\bar{x})$ when $P(\bar{c},\bar{b})\;$satisfies the Slater condition, and
there is no subset $D\subset T_{\bar{b}}(\bar{x})$ with$\;|D|<n\;$such that
$-\bar{c}\in \mathrm{cone}\left\{ a_{t},t\in D\right\}$.  \end{definition}

The following lemma is \cite[Theorem~2.1 and Lemma 3.1]{CanHanLopPar08}).

\begin{lemma}
\label{L6.1} Suppose that ENC is satisfied at $((\bar{c},\bar{b}),\bar{x})$. Then
\begin{enumerate}
\item
$\mathcal{S}$ is single valued and Lipschitz continuous in a
neighbourhood of $(\bar{c},\bar{b})$;
\item
if a sequence $\{((c_k,b_k),x_k)\}\subset\gph\mathcal{S}$ converges to $((\bar{c},\bar{b}),\bar{x})$, then $(b_k,x_k)\in \gph\mathcal{S}_{\bar{c}}$ for all $k$ large enough.
\end{enumerate}
\end{lemma}

Thanks to Lemma \ref{L6.1}, we can show that the parameter $c$ can be
considered fixed in our analysis, provided that ENC holds at $((\bar{c}%
,\bar{b}),\bar{x})$.

\begin{theorem}
If ENC is satisfied at $((\bar{c},\bar{b}),\bar{x})$, then
$\clmq\mathcal{S}((\bar{c},\bar{b}),\bar{x})=\clmq%
\mathcal{S}_{\bar{c}}(\bar{b},\bar{x})$.
\end{theorem}

\begin{proof}
It obviously holds $\clmq\mathcal{S}_{\bar{c}}(\bar{b},\bar{x})\geq \clmq\mathcal{S}((\bar{c},\bar{b}),\bar{x})$, and we need to show the opposite inequality.
If $\clmq\mathcal{S}((\bar{c},\bar{b}),\bar{x})=+\infty$, there is nothing to prove.
Let $\clmq\mathcal{S}((\bar{c},\bar{b}),\bar{x})<+\infty$.
Then
\sloppy
\begin{equation*}
\clmq\mathcal{S}((\bar{c},\bar{b}),\bar{x})=\lim_{k\rightarrow
+\infty }\frac{\Vert (c_k,b_k)-(\bar{c},\bar{b})\Vert }{\left\Vert
x_k-\bx\right\Vert^{q} }  \label{3.77}
\end{equation*}
for some sequence  $\{((c_k,b_k),x_k)\}\subset\gph\mathcal{S}$ such that $((c_k,b_k),x_k)\rightarrow((\bar{c},\bar{b}),\bx)$ and $x_k\neq \bar{x}$ for all $k\in\N$.
If ENC is satisfied at $((\bar{c},\bar{b}),\bar{x})$, then, by Lemma \ref{L6.1},
$(b_k,x_k)\in \gph\mathcal{S}_{\bar{c}}$ for all $k$ large enough.
Hence,
\begin{equation*}
\clmq\mathcal{S}((\bar{c},\bar{b}),\bar{x})\geq
\liminf_{k\to +\infty }\frac{\Vert b_k-\bar{b}\Vert _{\infty }%
}{\left\Vert x_k-\bar{x}\right\Vert^{q} }\geq \clmq\mathcal{S}%
_{\bar{c}}(\bar{b},\bar{x}).
\end{equation*}
This completes the proof.
\qed\end{proof}

\begin{example}
Consider the linear semi-infinite optimization problem in $\mathbb{R}^{2}$:
\begin{equation*}
\begin{aligned}
P(c,b):\quad& \text{minimize}
&&c_{1}x_{1}+c_{2}x_{2}\\
&\text{subject to} &&\left( \cos t\right) x_{1}+\left( \sin t\right)
x_{2}\leq b_{t},\;t\in[0,2\pi].
\end{aligned}
\end{equation*}
Let $\bar{c}:=(1,0)$ and $\bar{b}_{t}:=1,$ for all $t\in[0,2\pi].$
It is easy to check that $\bar{x}:=(-1,0)$ is the unique solution of $P\left( \bar{c},\bar{b}\right) .$
It obviously satisfies the Slater condition.
We are going to use condition (v) in Theorem~\ref{T4.2} to check the isolated calmness property of the solution mapping $\mathcal{S}$ of $P\left( \bar{c},\bar{b}\right) .$
The function \eqref{barf} takes the form
\begin{equation*}
f(x)=\max\left\{x_{1}+1,\sqrt{x_{1}^2+x_{2}^2}-1\right\}, \quad
x=(x_1,x_2)\in\R^2.
\end{equation*}
Obviously, $f(\bx)=0$ and, for any $x=(x_1,x_2)\in\R^2$ and $t>0$, we have
\begin{align*}
f(\bx+tx)&=\max\left\{tx_{1}, \sqrt{(tx_{1}-1)^2+(tx_{2})^2}-1\right\}
\\
&=t\max\left\{x_{1}, \frac{-2x_{1}+t(x_{1}^2+x_{2}^2)} {\sqrt{(tx_{1}-1)^2+(tx_{2})^2}+1}\right\}.
\end{align*}
Thus,
\begin{align}\label{E4.7-4}
f^{\prime}_{q}(\bar x;x)=\liminf_{(u_1,u_2)\to(x_1,x_2), \,t\downarrow0}t^{1-q}\max\left\{u_{1},-u_{1}+ \frac{t(u_{1}^2+u_{2}^2)} {2}\right\}\ge0.
\end{align}
If $x_1=0$ and $q<2$, then, by \eqref{E4.7-4},
\begin{align*}
f^{\prime}_{q}(\bar x;x)=\liminf_{u_1\to0, \,t\downarrow0}t^{1-q}|u_{1}|=0.
\end{align*}
and consequently, $\|f'_q(\bx;\cdot)\|_q=0$.
Thus, by Theorem~\ref{T4.2}, $\mathcal{S}$ does not possess $q$-or\-der isolated calmness property at $((\bar{c},\bar{b}),\bar{x})$ when $q<2$.
With $q=1$, this fact was established in \cite{CanDonLopPar09}.
Similarly, if $x_1=0$ and $q=2$, then
\begin{align*}
f^{\prime}_{2}(\bar x;x)=\liminf_{u_1\to0,\,t\downarrow0} \max\left\{\frac{u_1}{t},-\frac{u_1}{t}+ \frac{x_{2}^2} {2}\right\}=\inf_{\al\in\R}\max\left\{\al,-\al+ \frac{x_{2}^2}{2}\right\}=\frac{x_{2}^2}{4}.
\end{align*}
Finally, if $x_1\ne0$, it follows from \eqref{E4.7-4} that $f'_2(\bx;x)\ge\lim_{t\downarrow0}t\iv|x_1|=+\infty$.
Hence, $\|f'_2(\bx;\cdot)\|_2=\frac14>0$ and, by Theorem~\ref{T4.2}, $\mathcal{S}$ possesses $2$-order isolated calmness property at $((\bar{c},\bar{b}),\bar{x})$.
\qed\end{example}

\section{$q$-order sharp minimizers of $\ell_{p}$ penalty functions}

\label{S5}

In this section, we consider an inequality constrained optimization problem
\begin{gather}
\label{iop}\begin{aligned} & \text{minimize} && f(x)\\ &\text{subject to} &&g_i(x) \leq 0, \; i\in I:=\{1,\ldots,m\}, \end{aligned}
\end{gather}
where $f, g_{i} : \mathbb{R}^{n} \to\mathbb{R }\cup\{+\infty\}$, $i\in I$.
Given numbers $p > 0$ and $r > 0$, the $l_{p}$ penalty optimization problem
corresponding to (\ref{iop}) can be defined as follows:
\begin{equation}
\label{iopp}\text{minimize}\quad\ell_{p}(x) := f(x) + r \sum_{i=1}^{m}
(g^{+}_{i})^{p}(x),
\end{equation}
where ${g}^{+}_{i}(x):= \max\{0, g_{i}(x)\}$, $i \in I$.

By virtue of the optimal value function, relations between local minimizers of
(\ref{iop}) and (\ref{iopp}) were given in \cite{RubYan03,HuaYan03}. Below
$q$-order Hadamard directional subderivatives are used to identify $q$-order
sharp minimizers of the penalty problem (\ref{iopp}).

By Proposition \ref{P3.13} and Proposition \ref{P2.8}(iv), a point $\bar
x\in\cap_{i=1}^{m}\mathrm{dom}\, g_{i}\cap\mathrm{dom}\, f$ is a $q$-order
sharp minimizer of (\ref{iopp}) if and only if
\begin{equation}
\label{suff0}(\ell_{p})^{\prime}_{q}(\bar x;x) > 0 \quad\mbox{for all}\quad x
\ne0.
\end{equation}

Define
\begin{align*}
I(\bar x)  &  = \{ i \in I \ | \ g_{i}(\bar x) = 0 \},\\
K^{*}(\bar x)  &  = \{ x \in\mathbb{R}^{n} \ | \ \|x\| = 1, f^{\prime}%
_{q}(\bar x; x) \leq0\},\\
a(\bar x)  &  = \min\left\{  \sum_{i \in I(\bar x)} \left[  (g^{+}%
_{i})^{\prime}_{q/p}(\bar x; x) \right]  ^{p} \mid\ x \in K^{*}(\bar
x)\right\}  ,\\
b(\bar x)  &  = \min\{ f^{\prime}_{q}(\bar x; x) \ | \ \|x\| = 1\}.
\end{align*}

\begin{theorem} \label{T5.1}
\begin{enumerate}
\item
Suppose that $f'_q(\bx; \cdot)$ is proper and
\begin{equation} \label{suff}
f'_q(\bx; x) > 0
\qdtx{for all}
x \ne 0 \mbox{ with } (g^+_i)'_{q/p}(\bx; x) = 0, \ i \in I(\bx).
\end{equation}
Then $\bx$ is a $q$-order sharp minimizer of (\ref{iopp}) for all $r > \rho_0$,
where
\[ \rho_0:= \left\{ \begin{array}{ll} - b(\bx)/a(\bx), \ &\mbox{ if } K^*(\bx) \ne \emptyset,\\
0, \ &\mbox{ otherwise}. \end{array} \right. \]
\item
Suppose that $f$ is $q$-order Hadamard directionally
differentiable at $\bx$, $m=1$ and $\bx$ is a $q$-order sharp minimizer of (\ref{iopp}) for some $r>0$. Then (\ref{suff}) holds.
\end{enumerate}
\end{theorem}

\begin{proof}
\begin{enumerate}
\item
If $K^*(\bx) \ne \emptyset$, then $a(\bx) > 0$ and $b(\bx)  \leq 0.$
Therefore $\rho_0$ is well defined and nonnegative.
Let $r > \rho_0$.
Let $x \ne 0$.
As $(\ell_p)'_q(\bx; \cdot)$ is positively homogeneous, without loss of generality, we assume that $\|x\|=1$.
Obviously, $(g^+_i)'_{q/p}(\bx; x) \geq 0$ for all $i\in I$.
It is easy to show that
\begin{equation*} 
([g^{+}_i]^p)'_q(\bx;x) = \left[ (g^+_i)'_{q/p}(\bx; x)\right]^p
\qdtx{for all}
i \in I(\bx).
\end{equation*}
Since  $f'_q(\bx; \cdot)$ is proper, it follows that
\begin{equation} \label{suff2}
(\ell_p)'_q(\bx;x) \geq f'_q(\bx;x) + r \sum_{i \in I(\bx)} \left[(g^+_i)'_{q/p}(\bx; x) \right]^{p}.
\end{equation}
If $x \not\in K^*(\bx)$, we have $f'_q(\bx;x) > 0$, and consequently, $ (\ell_p)'_q(\bx;x) > 0$.
If $x \in K^*(\bx)$, then
$b(\bx) \leq 0$.
So, by definitions of $\rho_0$ and $a(\bx)$, we have
\[ r \sum_{i \in I(\bx)}  \left[(g^+_i)'_{q/p}(\bx; x)\right]^p > - b(\bx),\]
and thus it follows from (\ref{suff2}) that
\[ (\ell_p)'_q(\bx;x) > f'_q(\bx;x) -b(\bx) \geq 0.\]
So, by (\ref{suff0}), $\bx$ is a $q$-order sharp minimizer for (\ref{iopp}).
\item
It follows from Corollary \ref{P2.9} that (\ref{suff2}) holds as equality. The conclusion is verified by Proposition \ref{P3.13} and Proposition \ref{P2.8} (iv).
\qed\end{enumerate}
\end{proof}

Condition (\ref{suff}) in Theorem~\ref{T5.1} uses $(g^{+}_{i})^{\prime}%
_{q/p}(\bar x;\cdot)$, which allows the treatment of $q$-order sharp
minimizers of rather general penalty functions. When $p=q=1$,
Theorem~\ref{T5.1}(i) is a consequence of \cite[Theorem~4.1]{StuWar99}.

Furthermore, for all $i \in I(\bar x)$, $u\in\R^n$ and $t>0$, we obviously have
\begin{align*}
\frac{g^+_i(\bar x+tu)-g_i^+(\bar x)}{t^{q/p}}=&
\frac{g^+_i(\bar x+tu)}{t^{q/p}}=
\frac{\max \left\{0,g_i(\bar x+tu)\right\}}{t^{q/p}}\\=&
\max \left\{0,\frac{g_i(\bar x+tu)}{t^{q/p}}\right\}=
\max \left\{0,\frac{g_i(\bar x+tu)-g_i(\bar x)}{t^{q/p}}\right\}.
\end{align*}
Hence, if $g_{i}$ $(i \in I(\bar x))$ is $(q/p)$-order Hadamard directionally differentiable at $\bar x$, then so is $g_{i}^+$, and
\begin{equation*}
(g^{+}_{i})^{\prime}_{q/p} (\bar x;x) = \max\{ 0, (g_{i})^{\prime}_{q/p}(\bar
x;x)\}
\end{equation*}
for all $x\in\R^n$.
Therefore, if all $g_{i}$ $(i \in I(\bar x))$ are $(q/p)$-order Hadamard directionally differentiable at $\bar x$, then (\ref{suff}) is equivalent
to the following condition:
\[
f^{\prime}_{q}(\bar x; x) > 0 \quad\mbox{for all}\quad x \ne0 \mbox{ with }
(g_{i})^{\prime}_{q/p}(\bar x; x) \leq0, \ i \in I(\bar x).
\]

The following simple example shows the calculation of the exact upper bound of
the 1-order sharp minimizer of the penalty problem (\ref{iopp}).

\begin{example}
Consider the following problem on $\R$:
\begin{gather*}
\text{minimize  }x
\quad
\text{subject to  } x^{2s} \leq 0,
\end{gather*}
where $s > 0$.
Obviously
$\bx=0$ is a minimizer of this problem.
With any $p>0$,
we have
\[ \mbox{shrp}_1 \ell_{p}(0) =
\left\{
\begin{array}{ll} + \infty, & \mbox{ if } sp<1/2,\\
r+1, & \mbox{ if } sp=1/2,\\
0, & \mbox{ otherwise. }
\end{array}  \right.
\]
\end{example}

\section*{Acknowledgement}

The authors wish to thank the referees and the handling editor for their careful reading of the manuscript and valuable comments and suggestions, particularly concerning the references.
The very constructive referee reports have helped us improve the presentation.

\addcontentsline{toc}{section}{References}

\bibliographystyle{spmpsci}
\bibliography{buch-kr,kruger,kr-tmp}

\end{document}